\documentclass[a4paper,11pt]{article}

\usepackage{xcolor}   
\usepackage{amsmath}  
\usepackage{amssymb}   
\usepackage{amsfonts}  
\usepackage{mathrsfs}  
\usepackage{amsthm}    
\usepackage{indentfirst}
\usepackage[title]{appendix}
\usepackage{hyperref}  
\hypersetup{hypertex=true,
	colorlinks=true,
	linkcolor=blue,
	anchorcolor=blue,
	citecolor=blue}
\usepackage{microtype}  
\newtheorem{theorem}{Theorem}[section]

\newtheorem{proposition}[theorem]{Proposition}
\newtheorem{lemma}[theorem]{Lemma}
\newtheorem{remark}{Remark}

\numberwithin{equation}{section}
\allowdisplaybreaks[4]

\begin{document}
	\pagenumbering{arabic}
	\bigskip\bigskip
	\noindent{\Large\bf Precise large deviations for the total population of heavy-tailed subcritical branching process with immigration
		\footnote{ This work was supported in part by NSFC (NO. 11971062) and  the National Key Research and Development Program of China (No. 2020YFA0712900).} }
	
	\noindent
	{Jiayan Guo\footnote{ School of Mathematical Sciences \& Laboratory of Mathematics and Complex Systems, Beijing Normal University, Beijing 100875, P.R. China. Email: guojiayan@mail.bnu.edu.cn}
		\quad 
		Wenming Hong\footnote{ School of Mathematical Sciences \& Laboratory of Mathematics and Complex Systems, Beijing Normal	University, Beijing 100875, P.R. China. Email: wmhong@bnu.edu.cn}
	}
	
	\begin{center}
		\begin{minipage}{12cm}
			\begin{center}\textbf{Abstract}\end{center}
			\footnotesize
			
			In this article we focus on the partial sum $S_{n}=X_{1}+\cdots+X_{n}$ of the subcritical branching process with immigration $\{X_{n}\}_{n\in\mathbb{N_{+}}}$, under the condition that one of the offspring $\xi$ or immigration $\eta$ is regularly varying. The tail distribution of $S_n$ is heavily dependent on that of $\xi$ and $\eta$, and a precise large deviation probability for $S_{n}$ is specified. (i)When the tail of offspring $\xi$ is ``lighter" than immigration $\eta$,   uniformly for $x\geq x_{n}$,  $P(S_{n}-ES_{n}>x)\sim c_{1}nP(\eta>x)$ with some constant $c_{1}$ and sequence $\{x_{n}\}$,   where $c_{1}$ is only related to the mean of offspring; (ii) When the tail of immigration $\eta$ is not ``heavier" than offspring $\xi$,    uniformly for $x\geq x_{n}$,   $P(S_{n}-ES_{n}>x)\sim c_{2}nP(\xi>x)$  with some constant $c_{2}$ and sequence $\{x_{n}\}$, where  $c_{2}$ is related to both the mean of offspring and the mean of immigration. 			
			
		\end{minipage}
	\end{center}

	\section{Introduction}
	Let $\{X_{n}\}$ be a branching process with immigration which is defined by $X_{0}=0$ and  
	\begin{equation}\label{defGWI1}
		X_{n}=\sum_{i=1}^{X_{n-1}}\xi_{n,i}+\eta_{n}, \quad n\in\mathbb{N_{+}}, 
	\end{equation}
	(with the convention $\sum_{i=1}^{0} =0$), where $\{\xi_{n,i}\}_{n,i\in\mathbb{N_{+}}}$ and $\{\eta_{n}\}_{n\in\mathbb{N_{+}}}$ are two independent i.i.d sequences of nonnegative integer-valued random variables. To exclude trivialities, we always assume that $P(\eta=0)<1$. Use $\xi$, $\eta$ for the generic copies and $\alpha:=E\xi$, $\beta:=E\eta$ for their means, respectively. When $\alpha:=E\xi<1(=1, >1)$, we say the process is subcritical (critical, supercritical). In this paper we consider the subcritical case.
	
	To ease notation, we introduce the i.i.d random operator $\theta_{n}(n\in\mathbb{N_{+}})$ as
	$$\theta_{n}\circ k=\sum_{i=1}^{k}\xi_{n,i}, \quad k\in\mathbb{N},$$ where $\theta_{n}\circ 0=0$. And $\theta_{n}\circ(k_{1}+k_{2})\overset{\text{d}}{=}\theta_{n}^{(1)}\circ k_{1}+\theta_{n}^{(2)}\circ k_{2}$, where $\theta_{n}^{(1)}$ and $\theta_{n}^{(2)}$ on the right-hand side are independent with the same distribution as $\theta_{n}$. Then (\ref{defGWI1}) can be written as
	\begin{equation}\label{defGWI2}
		X_{n}=\theta_{n}\circ X_{n-1}+\eta_{n}, \quad n\in\mathbb{N_{+}}. 
	\end{equation}
	
	The limiting behavior of the subcritical process has attracted much attention in literature. It is shown in Foster and Williamson \cite{Foster} that $\{X_{n}\}_{n\in\mathbb{N}}$ has a stationary distribution $X$ if and only if
	\begin{equation}\label{condition}
		E\log^{+}\eta=\sum_{k=1}^{\infty}P(\eta=k)\log k<\infty.
	\end{equation} 
	Kevei and Wiandt \cite{Kevei} gave a necessary and sufficient condition for the existence of moments of $X$ in multi-type case. As for the tail distribution, Basrak et.al \cite{Basrak} proved that $X$ is also regularly varying  when $\xi$ or $\eta$ is regularly varying , and Foss and Miyazawa \cite{Foss} extended their results onto the more general case.

	We are interested in the large deviation of probabilities $P(S_{n}>x)$ for the partial sum of the process, where, 
	$$S_{n}=X_{1}+\cdots+X_{n}.$$	
	When Cram\'{e}r's condition is satisfied, namely, for some $\theta>0$, $Ee^{\theta\xi}<\infty$ and $Ee^{\theta\eta}<\infty$, Shihang Yu et.al \cite{Yushihang} have provided the exact form of large and moderate deviations for the empirical mean of population ${S_{n}}/{n}$ and centered total population $S_{n}-ES_{n}$, where the rate functions are explicitly identified, by analyzing the relation between $X_{n}$ and $X_{n-1}$  and verifying the conditions of G\"{a}rtner-Ellis theorem.

	In the present paper, we focus on the case when Cram\'{e}r's condition is not satisfied, for example the distribution of $\xi$ or $\eta$ is heavy-tailed, or precisely speaking, regularly varying, the behavior of large deviation probability $P(S_{n}>x)$. We will identify it in what follows, by decomposing $S_{n}$ and using the properties of regularly varying functions and the stationary distribution $X$.\\

	For convenience, we summarize some known results on regularly varying  distribution in Appendix A and the precise large deviation results for i.i.d regularly varying sequence in Appendix B. 	
	Throughout this paper, $f(x)=o(g(x))$ means $\lim_{x\rightarrow\infty}f(x)/g(x)=0$, and $f(x)\sim g(x)$ means $\lim_{x\rightarrow\infty}f(x)/g(x)=1$, for two vanishing (at infinity) functions. When the value of a positive constant is not of interest, we write $c$ for them.\\

	Two kind of models are considered. 
	
	In the first model, we assume that $\eta$ is regularly varying, and the tail of $\xi$ is lighter, i.e.,
	\begin{equation}\label{conditionA1}\tag{A1}
		0<\alpha=E\xi<1,
	\end{equation}
	\begin{equation}\label{conditionA2}\tag{A2}
		P(\eta>x)=x^{-\kappa}L(x), 
	\end{equation}
	for some $\kappa>0$ and a slowly varying function $L(x)$. For $\kappa\geq1$, we also assume that
	\begin{equation}\label{conditionA3}\tag{A3}
		\exists\delta>0,\quad E(\xi^{\kappa+\delta})<\infty.
	\end{equation}
	
	Then as $x\rightarrow\infty$,
	$$P(\sum_{i=1}^{\eta}\xi_{i}>x)\sim (E\xi)^{\kappa}P(\eta>x)$$
	by Lemma \ref{refA2}, which implies that $\xi_{1}+\xi_{2}+\cdots+\xi_{\eta}$ will inherit the regular variation dominated by $\eta$. Note that condition (\ref{condition}) is fulfilled under (\ref{conditionA2}), then as a consequence, there exists a  stationary distribution $X$ for the sequence  $\{X_{n}\}_{n\in\mathbb{N}}$.	We prove that $X$ is regularly varying with index $\kappa$ in this model in Lemma \ref{lemmaA}, i.e.,
	\begin{equation*}
		\lim_{x\rightarrow\infty}\frac{P(X>x)}{P(\eta>x)}=\frac{1}{1-\alpha^{\kappa}},
	\end{equation*}
	using the similar method as Theorem 2.1.1 in Basrak et.al \cite{Basrak} (there the second moment for $\xi$ is needed when $\kappa\in[1,2)$;  but here we improve the conditions as (\ref{conditionA3})). Then by decomposing $S_{n}$ and using the properties of regularly varying functions, we get in Proposition \ref{theoremA0} that for fixed $n$, $S_{n}$ is also regularly varying, i.e., 
	\begin{equation*}
		\lim_{x\rightarrow\infty}\frac{P(S_{n}>x)}{P(\eta>x)}
		=\sum_{i=1}^{n}[(\sum_{m=0}^{i-1}\alpha^{m})^{\kappa}].
	\end{equation*}

	Furthermore we can couple the increase of $x$ with $n$ to obtain probabilities of precise large deviations uniformly for $x\geq x_{n}$, where $\{x_{n}\}$ is some appropriate sequences tend to infinity. We have the following result,

	\begin{theorem}\label{theoremA}		
		Assume (\ref{conditionA1})-(\ref{conditionA3}) are satisfied, then there exits sequence $\{x_{n}\}\uparrow\infty$ that
		\begin{equation}\label{eqtheoremA1}			
			\lim_{n\rightarrow\infty}\sup_{x\geq  x_{n}}\left|\frac{P(S_{n}-d_{n}>x)}{nP(\eta>x)}-\frac{1}{(1-\alpha)^{\kappa}}\right|=0
		\end{equation}
		and
		\begin{equation}\label{eqtheoremA2}
			\lim_{n\rightarrow\infty}\sup_{x\geq x_{n}}\frac{P(S_{n}-d_{n}\leq-x)}{nP(\eta>x)}=0,
		\end{equation}
		where
		\begin{equation*}
			d_{n}=\left\{
			\begin{aligned}
				\begin{array}{cl}
					0, &\kappa\in(0,1]\\
					ES_{n}, &\kappa\in(1,\infty)
				\end{array}
			\end{aligned}
			\right.
		\end{equation*}
		and if $\kappa\in(0,2]$, one can choose $x_{n}=n^{\delta+1/\kappa}$ for any $\delta>0$; if $\kappa\in(2,\infty)$, one can choose $x_{n}=\sqrt{an\log n}$ for $a>\kappa-2$.\\
	\end{theorem}

	In the second model, we assume that $\xi$ is regularly varying, and the tail of $\eta$ is lighter or comparable with $\xi$, i.e.,
	\begin{equation}\label{conditionB1}\tag{B1}
		0<\alpha=E\xi<1,
	\end{equation}
	\begin{equation}\label{conditionB2}\tag{B2}
		P(\xi>x)=x^{-\kappa}L(x),
	\end{equation}
	for some $\kappa>1$ and a slowly varying function $L(x)$. 	
	And one of the following  conditions is satisfied:
	\begin{equation}\label{conditionB3}\tag{B3}
		(i)\,\exists\,\delta>0,\, E(\eta^{\kappa+\delta})<\infty\,;
	\end{equation}
	\begin{equation}\label{conditionB4}\tag{B4}
		(ii)\,\exists\, p>0,\, P(\eta>x)=x^{-\kappa}L_{1}(x)\sim pP(\xi>x)\,,
	\end{equation}
	where $L_{1}(x)$ is also a slowly varying function. 
	
	Since \eqref{conditionB3} implies $P(\eta>x)=o(P(\xi>x))$, we denote $p=0$ in this case. Then by Lemma \ref{refA3} and Lemma \ref{refA4} we have, as $x\rightarrow\infty$,
	$$P(\sum_{i=1}^{\eta}\xi_{i}>x)\sim E\eta P(\xi>x)+ p(E\xi)^{\kappa}P(\xi>x)$$
	for $p\geq0$, which means that $\xi_{1}+\xi_{2}+\cdots+\xi_{\eta}$ will inherit the regular variation, from both $\xi$ and $\eta$. 	
	It is shown in Lemma \ref{lemmaB} and Proposition \ref{theoremB0} that $X$ and $S_{n}$ is regularly varying with index $\kappa$ in this model. We also prove the following large deviation result,
	\begin{theorem}\label{theoremB}
		Assume (\ref{conditionB1})-(\ref{conditionB4}) are satisfied, then there exits sequence $\{x_{n}\}\uparrow\infty$ that
		\begin{equation}\label{eqtheoremB1}
			\lim_{n\rightarrow\infty}\sup_{x\geq 	x_{n}}\left|\frac{P(S_{n}-ES_{n}>x)}{nP(\xi>x)}-\frac{\beta+p(1-\alpha)}{(1-\alpha)^{\kappa+1}}\right|=0
		\end{equation}
		and
		\begin{equation}\label{eqtheoremB2}
			\lim_{n\rightarrow\infty}\sup_{x\geq x_{n}}\frac{P(S_{n}-ES_{n}\leq-x)}{nP(\xi>x)}=0,
		\end{equation}		
		and if $\kappa\in(1,2]$, one can choose $x_{n}=n^{\delta+1/\kappa}$ for any $\delta>0$; if $\kappa\in(2,\infty)$, one can choose $x_{n}=\sqrt{an\log n}$ for $a>\kappa-2$.
	\end{theorem}

	\begin{remark}	
		\rm{	
			For the  summation of i.i.d random variable with regularly varying distribution, the precise large deviations have been considered by many authors, see for example, Heyde\cite{Heyde}, Nagaev.A.V\cite{AV}, Nagaev.S.V\cite{SV}, Cline and Hsing \cite{Cline}, which we summarize in  Theorem \ref{refB}.
			
			The situation is different for the partial sums $S_{n}=X_{1}+\cdots+X_{n}$ of the branching processes with immigration because of the dependent structure of the sequence $\{X_n\}$, which reflects on the rate constant respectively. (i) When the tail of offspring $\xi$ is ``lighter" than the immigration $\eta$  (the first model), $S_{n}$ is regularly varying as well with the same index of $\eta$, and  with some constant $c_{1}$ and sequence $x_{n}$, uniformly for $x\geq x_{n}$,  $P(S_{n}-d_{n}>x)\sim c_{1}nP(\eta>x)$, where $c_{1}$ is only related to the mean of the offspring. (ii) When the tail of the immigration $\eta$ is not ``heavier" than the offspring $\xi$ (the second model), $S_{n}$ is regularly varying as well with the same index of $\xi$, and with some constant $c_{2}$ and sequence $x_{n}$, uniformly for $x\geq x_{n}$, $P(S_{n}-ES_{n}>x)\sim c_{2}nP(\xi>x)$ where $c_{2}$ is related to both the mean of the offspring and the mean of the immigration. 		
	}
	\end{remark}
	
	\begin{remark}
		\rm{
			For the summation of independent but not identically regularly varying distributed random variables, we refer to Paulauskas and Sku\v{c}ait\.{e} \cite{Skuchaite}, who proved a large deviation result under the condition that the average of distribution functions of these random variables is equivalent to some regularly varying limit distribution function with index $\kappa>1$. In our proof, although $S_{n}$ can be divided into $n$ independent but not identically distributed random variables, but in one hand note that it may appear that $\kappa\in(0,1)$ in our first model, which is not contained in \cite{Skuchaite}, and on the other hand the method is different: we will prove our results by using the branching properties and the limiting behavior of the process.
		}
	\end{remark}
	
	\begin{remark}	
		\rm{	
			For the total population of branching process with immigration in random environment, we refer to Buraczewski and Dyszewski \cite{Buraczewski18}. They established precise large deviations in the nearest neighbour random walk in random environment, which can be seen as the subcritical branching process with single immigration in random environment. However, it is required in \cite{Buraczewski18} that $E\log A<0$ and there exists $\kappa>0$, s.t. $EA^{\kappa}=1$, where $A$ is the quenched mean of the offspring, and this cannot be degenerated  to our model. }
	\end{remark}

	\begin{remark}	
		\rm{	
			For the  solutions to stochastic recurrence equations $Y_{n}=A_{n}Y_{n-1}+B_{n}\, (n\in \mathbb{Z})$, which can be understood as the quenched mean of branching process with immigration in random environment, it is shown in Kesten \cite{Kesten} and Goldie \cite{Goldie} that if $E\log A<0$, $E\log^{+} B<\infty$ then the equation has a unique and strictly ergodic solution $(Y_{i})$. For the stationary sequence $(Y_{i})$, when Kesten's condition is satisfied, Buraczewski et.al \cite{Buraczewski} proved precise large deviations for partial sum of the stationary sequence $Y_{1}+\cdots+Y_{n}$; when Kesten's conditions are not satisfied and $B$ is regularly varying, precise large deviations were given by Konstantinides and Mikosch\cite{Dimitrios}. Although (\ref{defGWI2}) is somewhat similar in form to the stochastic recurrence equation, it is actually convolution rather than multiplication.}
	\end{remark}

	The article is organized as follow. In Section 2 we analyze the moments and regular variation of underlying branching process without immigration. In Section 3 we study the tail behavior of stationary distribution, which is also regularly varying. In Section 4 we give the regular variation of $S_{n}$ and prove the main results Theorem \ref{theoremA} and Theorem \ref{theoremB}, for large deviations of the partial sum.  Some basic facts needed in the proof are listed in appendix.

	\section{Moments and regular variation of underlying process}
	Let $\{Z_{n}\}$ be the underlying subcritical branching process (without immigration), which is defined by $Z_{0}=1$ and the same offspring distribution as $\{X_{n}\}$, i.e.,
	\begin{equation*}
		Z_{n}=\sum_{i=1}^{Z_{n-1}}\xi_{n,i}, \quad n\in\mathbb{N_{+}}.
	\end{equation*}
	
	Let
	\begin{equation}\label{defTn}
		T_{n}:=1+Z_{1}+\cdots+Z_{n}
	\end{equation}
	be the total population of $\{Z_{n}\}$ up to the $n$th generation, and
	\begin{equation}\label{defT}
		T:=1+Z_{1}+\cdots+Z_{n}+\cdots
	\end{equation}
	be the total population of $\{Z_{n}\}$.
	
	Since $EZ_{n}=\alpha^{n}$, we have $$ET_{n}=1+\alpha+\alpha^{2}+\cdots+\alpha^{n}=\frac{1-\alpha^{n}}{1-\alpha}<\infty.$$
	
	Using the branching property,
	$$T\overset{d}{=}1+\sum_{i=1}^{\xi}T^{(i)},$$
	where $\{T^{(i)}\}_{i}$ is i.i.d and have the same distribution as $T$,
	then
	\begin{equation*}
		ET=\frac{1}{1-\alpha}<\infty.
	\end{equation*}
	
	For higher moments of $T_{n}$ and $T$, we have the following lemma.
	\begin{lemma}\label{lemmaGW1}
		If $\alpha<1$ and $E(\xi^{h})<\infty$ for some $h>1$, then $\forall n\in\mathbb{N_{+}}$, $E(T_{n}^{h})\leq E(T^{h})<\infty$.
	\end{lemma}
	\begin{proof}
		Actually this is obtained by \cite{Kevei} in the proof of moments of the stationary distribution of subcritical multi-type branching process with immigration. If additionally $E\eta^{h}<\infty$, then it is shown in (11) of \cite{Kevei} that there exists constants $0<v<1$ and $c>0$,
		$$E(\theta_{1}\circ\theta_{2}\circ\cdots\circ\theta_{k}\circ\eta)^{h}\leq cv^{k},\, \forall k\in\mathbb{N}.$$
		
		Take $\eta\equiv1$, then we have $E(Z_{k}^{h})\leq cv^{k}$, and the result follows by Minkowski's inequality,
		$$
		[E(T^{h})]^{\frac{1}{h}}
		=[E(\sum_{n=0}^{\infty}Z_{n})^{h}]^{\frac{1}{h}}
		\leq \sum_{n=0}^{\infty}[E(Z_{n}^{h})]^{\frac{1}{h}}
		\leq \sum_{n=0}^{\infty}(cv^{n})^{\frac{1}{h}}<\infty.
		$$
	\end{proof}
	
	The next three lemmas imply that, in the second model, when $\xi$ is regularly varying, $Z_{n}$ is also regularly varying, as well as $T_{n}$ and $T$.
	\begin{lemma}\label{lemmaGW2}
		If $\alpha<1$ and $\xi$ is regularly varying with $\kappa>1$, then $\forall n\in\mathbb{N_{+}}$, as $x\rightarrow\infty,$
		\begin{align*}
			P(Z_{n}>x)\sim\frac{\alpha^{n}-\alpha^{\kappa n}}{\alpha-\alpha^{\kappa}}P(\xi>x).
		\end{align*}
	\end{lemma}
	\begin{proof}
		We will prove it by induction. Obviously it is true for $n=1$. If it is true for some $n\geq 1$, then for $n+1$, 
		\begin{align*}
			P(Z_{n+1}>x)
			&=P(\sum_{i=1}^{Z_{n}}\xi_{n+1,i}>x)\\
			&\sim EZ_{n} P(\xi>x)+ P(Z_{n}>\frac{x}{\alpha})\\
			&=\alpha^{n}P(\xi>x)
			+\alpha^{\kappa}\frac{\alpha^{n}-\alpha^{\kappa n}}{\alpha-\alpha^{\kappa}}P(\xi>x)\\
			&=\frac{\alpha^{n+1}-\alpha^{\kappa (n+1)}}{\alpha-\alpha^{\kappa}}P(\xi>x),
		\end{align*}
		the second step is by Lemma \ref{refA4}, and notice that $EZ_{n}=\alpha^{n}$.
	\end{proof}

	\begin{lemma}\label{lemmaGW3}
		If $\alpha<1$ and $\xi$ is regularly varying with $\kappa>1$, then $\forall n\in\mathbb{N_{+}}$, as $x\rightarrow\infty,$
		\begin{align*}
			P(T_{n}>x)\sim\sum_{i=0}^{n-1}\alpha^{i}(\frac{1-\alpha^{n-i}}{1-\alpha})^{\kappa}P(\xi>x).
		\end{align*}
	\end{lemma}
	\begin{proof} Again, we will prove it by induction. Obviously the lemma is true for $n=1$, where
		$$P(T_{1}>x)=P(\xi>x-1)\sim P(\xi>x).$$
		
		If it is true for some $n\geq 1$, then for $n+1$, using the branching property with $\{T_{n}^{(i)}\}_{i}$ independent and having the same distribution as $T_{n}$, we have
		\begin{align*}
			P(T_{n+1}>x)
			&=P(1+\sum_{i=1}^{\xi}T_{n}^{(i)}>x)\\
			&=P(\sum_{i=1}^{\xi}T_{n}^{(i)}>x-1)\\
			&\sim E\xi P(T_{n}>x-1)+P(\xi>\frac{x-1}{ET_{n}})\\
			&\sim\alpha\sum_{i=0}^{n-1}\alpha^{i}(\frac{1-\alpha^{n-i}}{1-\alpha})^{\kappa}P(\xi>x-1)
			+(\frac{1-\alpha^{n+1}}{1-\alpha})^{\kappa}P(\xi>x-1)\\
			&\sim\left[\alpha\sum_{i=0}^{n-1}\alpha^{i}(\frac{1-\alpha^{n-i}}{1-\alpha})^{\kappa}
			+(\frac{1-\alpha^{n+1}}{1-\alpha})^{\kappa}\right]P(\xi>x)\\
			&=\sum_{i=0}^{n}\alpha^{i}(\frac{1-\alpha^{n+1-i}}{1-\alpha})^{\kappa}P(\xi>x)
		\end{align*}
		by Lemma \ref{refA4} and $ET_{n}=1+\alpha+\alpha^{2}+\cdots+\alpha^{n}$.
	\end{proof}
	
	\begin{lemma}\label{lemmaGW4}
		If $\alpha<1$ and $\xi$ is regularly varying with $\kappa>1$, then as $x\rightarrow\infty,$
		\begin{align*}
			P(T>x)\sim\frac{1}{(1-\alpha)^{\kappa+1}}P(\xi>x).
		\end{align*}
	\end{lemma}
	\begin{proof}
		By branching property, we have
		$$T\overset{\text{d}}{=}1+\sum_{i=1}^{\xi}T^{(i)},$$
		where $\{T^{(i)}\}_{i}$ is i.i.d and has the same distribution as $T$.
		
		Then by Lemma \ref{refA5}, as $x\rightarrow\infty$,
		\begin{align*}
			P(T>x)\sim\frac{1}{1-\alpha}\cdot P(1+ET\cdot\xi>x)\sim\frac{1}{(1-\alpha)^{\kappa+1}}P(\xi>x).
		\end{align*}
	\end{proof}

	\section{Regular variation of stationary distribution}	
	Recall that $X$ is the stationary distribution of $\{X_{n}\}$. Define a sequence of independent random variables $C_{0}:=\eta_{0}$ and 
	$$C_{n}:=\theta_{n}^{(n)}\circ\theta_{n-1}^{(n)}\circ\cdots\circ\theta_{1}^{(n)}\circ\eta_{n},\quad n\geq1,$$
	where $\eta_{n}(n\geq0)$ are independent with the same distribution as $\eta$, $\theta_{i}^{(n)}(n\geq1)$ are independent with the same distribution as $\theta_{i}$, then
	$$X\overset{d}{=}\sum_{n=0}^{\infty}C_{n}.$$
	
	The next two results of the tail distribution of $X$ are proved in Basrak et.al\cite{Basrak}, but under more restrictive moment conditions in the first model. Here we give a proof similar but under lower moments of $\xi$ by using Lemma \ref{lemmaGW1}.
	
	\begin{lemma}\label{lemmaA}
		Assume (\ref{conditionA1})-(\ref{conditionA3}) are satisfied, then as $x\rightarrow\infty$,
		\begin{equation*}
			P(X>x)\sim\frac{1}{1-\alpha^{\kappa}}P(\eta>x).
		\end{equation*}
	\end{lemma} 
	\begin{proof}
		The case $\kappa\in(0,1)$ is same as Basrak et.al\cite{Basrak}, so we only consider $\kappa\geq1$.

		By Lemma \ref{refA2}, for each $n\in\mathbb{N_{+}}$,
		$$P(C_{n}>x)\sim P(\eta>\frac{x}{EZ_{n}})=\alpha^{\kappa n}P(\eta>x),$$
		which means each $C_{n}$ is regularly varying with index $\kappa$. Since they are independent, by Lemma \ref{refA1}, the finite summation is also regularly varying, i.e.,
		$$\lim_{x\rightarrow\infty}\frac{P(\sum_{i=0}^{n}C_{i}>x)}{P(\eta>x)}=\frac{1-\alpha^{\kappa(n+1)}}{1-\alpha^{\kappa}}.$$
		
		Thus we have the lower bound
		\begin{equation*}
			\varliminf_{x\rightarrow\infty}\frac{P(X>x)}{P(\eta>x)}
			\geq
			\lim_{n\rightarrow\infty}\lim_{x\rightarrow\infty}\frac{P(\sum_{i=0}^{n}C_{i}>x)}{P(\eta>x)}
			=\frac{1}{1-\alpha^{\kappa}}.
		\end{equation*}
		
		As for the upper bound
		\begin{equation*}
			\varlimsup_{x\rightarrow\infty}\frac{P(X>x)}{P(\eta>x)}\leq
			\frac{1}{1-\alpha^{\kappa}},
		\end{equation*}			
		notice that $\forall\varepsilon>0$, $\forall n\in\mathbb{N_{+}}$,
		\begin{align*}
			P(X>x)\leq P(\sum_{i=0}^{n}C_{i}>(1-\varepsilon)x)
			+P(\sum_{i=n+1}^{\infty}C_{i}>\varepsilon x).
		\end{align*}
		
		Write $\tilde{\kappa}:=\kappa+\frac{1}{2}\min\{1,\delta\}$, where $\delta$ is define in \eqref{conditionA3}. Then it is sufficient to show
		\begin{equation}\label{eq3A1}
			\lim_{n\rightarrow\infty}\varlimsup_{x\rightarrow\infty}
			\frac{P(\sum_{i=n}^{\infty}C_{i}>x^{1/\tilde{\kappa}})}{P(\eta>x^{1/\tilde{\kappa}})}=\lim_{n\rightarrow\infty}\varlimsup_{x\rightarrow\infty}
			\frac{P((\sum_{i=n}^{\infty}C_{i})^{\tilde{\kappa}}>x)}{P(\eta^{\tilde{\kappa}}>x)}=0.
		\end{equation}

		Now we focus on proving \eqref{eq3A1}. 
		
		Recall that under condition \eqref{conditionA3} and $\tilde{\kappa}<\kappa+\delta$, it is shown in Lemma \ref{lemmaGW1} that there exits constants $c>0$ and $v<1$, such that $\forall i\in\mathbb{N_{+}}$,
		$$E(Z_{i}^{\tilde{\kappa}})<cv^{i}.$$
		
		Notice 
		\begin{equation}\label{eq3A2}
			\begin{aligned}
				\frac{P((\sum_{i=n}^{\infty}C_{i})^{\tilde{\kappa}}>x)}{P(\eta^{\tilde{\kappa}}>x)}
				\leq\frac{P(\cup_{i\geq n}\{\eta_{i}^{\tilde{\kappa}}>x/v^{i\tilde{\kappa}}\})}{P(\eta^{\tilde{\kappa}}>x)}
				+\frac{P\{(\sum_{i=n}^{\infty}C_{i}1_{\{\eta_{i}^{\tilde{\kappa}}\leq x/v^{i\tilde{\kappa}}\}})^{\tilde{\kappa}}>x\}}{P(\eta^{\tilde{\kappa}}>x)}.
			\end{aligned}
		\end{equation}
		
		The first term of \eqref{eq3A2} can be bound by
		\begin{equation}\label{eq3A3}
			\begin{aligned}
				\frac{P(\cup_{i\geq 	n}\{\eta_{i}^{\tilde{\kappa}}>x/v^{i\tilde{\kappa}}\})}{P(\eta^{\tilde{\kappa}}>x)}\leq\sum_{i=n}^{\infty}\frac{P(\eta^{\tilde{\kappa}}>x/v^{i\tilde{\kappa}})}{P(\eta^{\tilde{\kappa}}>x)},
			\end{aligned}		
		\end{equation}
		and since $$P(\eta^{\tilde{\kappa}}>x)=P(\eta>x^{1/\tilde{\kappa}})=x^{-\kappa/\tilde{\kappa}}L(x^{1/\tilde{\kappa}}),$$
		the random variable $\eta^{\tilde{\kappa}}$ is regularly varying with index $\kappa/\tilde{\kappa}\in(0,1).$ So using Potter's bound (Lemma \ref{Potter}), for any chosen $A>1$, $B=\kappa/2\tilde{\kappa}>0$, there exists $X=X(A,B)$ such that for all $x\geq X, x/v^{i\tilde{\kappa}}\geq X$,
		\begin{align*}
			\frac{P(\eta^{\tilde{\kappa}}>x/v^{i\tilde{\kappa}})}{P(\eta^{\tilde{\kappa}}>x)}\leq A\max\{(\frac{1}{v^{i\tilde{\kappa}}})^{-\frac{\kappa}{\tilde{\kappa}}+\frac{\kappa}{2\tilde{\kappa}}},(\frac{1}{v^{i\tilde{\kappa}}})^{-\frac{\kappa}{\tilde{\kappa}}-\frac{\kappa}{2\tilde{\kappa}}}\}=Av^{i\kappa/2},
		\end{align*}
		so the first term tends to zero by first letting $x\rightarrow\infty$ and  then letting $n\rightarrow\infty$ in \eqref{eq3A3}.

		The second term of \eqref{eq3A2} can be bound by
		\begin{equation}\label{eq3A4}		
			\begin{aligned}
				\left[\frac{P\{(\sum_{i=n}^{\infty}C_{i}1_{\{\eta_{i}^{\tilde{\kappa}}\leq x/v^{i\tilde{\kappa}}\}})^{\tilde{\kappa}}>x\}}{P(\eta^{\tilde{\kappa}}>x)}\right]^{1/\tilde{\kappa}}
				&\leq
				\left[\frac{E(\sum_{i=n}^{\infty}C_{i}1_{\{\eta_{i}^{\tilde{\kappa}}\leq x/v^{i\tilde{\kappa}}\}})^{\tilde{\kappa}}}{xP(\eta^{\tilde{\kappa}}>x)}\right]^{1/\tilde{\kappa}}\\
				&\leq
				\sum_{i=n}^{\infty}\left[\frac{E(C_{i}^{\tilde{\kappa}}1_{\{\eta_{i}^{\tilde{\kappa}}\leq x/v^{i\tilde{\kappa}}\}})}{xP(\eta^{\tilde{\kappa}}>x)}\right]^{1/\tilde{\kappa}},
			\end{aligned}
		\end{equation}
		by using Markov inequality and Minkowski's inequality.
		
		Denote $Z_{i}^{(j)}:=\theta_{i}^{(j)}\circ\theta_{i-1}^{(j)}\circ\cdots\circ\theta_{1}^{(j)}\circ1$, then $\{Z_{i}^{(j)}\}_{j}$ are independent and have the same distribution as the underlying process $Z_{i}$, and
		\begin{equation}\label{eq3A5}
			\begin{aligned}
				E(C_{i}^{\tilde{\kappa}}1_{\{\eta_{i}^{\tilde{\kappa}}\leq x/v^{i\tilde{\kappa}}\}})
				=&\sum_{m^{\tilde{\kappa}}\leq x/v^{i\tilde{\kappa}}}
				E(C_{i}^{\tilde{\kappa}}1_{\{\eta_{i}=m\}})\\	
				=&\sum_{m^{\tilde{\kappa}}\leq x/v^{i\tilde{\kappa}}}
				E(\sum_{j=1}^{m}Z_{i}^{(j)})^{\tilde{\kappa}}P(\eta=m)\\
				\leq&\sum_{m^{\tilde{\kappa}}\leq x/v^{i\tilde{\kappa}}}
				\{\sum_{j=1}^{m}[E(Z_{i}^{^{\tilde{\kappa}}})]^{1/\tilde{\kappa}}\}^{\tilde{\kappa}}P(\eta=m)\\
				\leq&\sum_{m^{\tilde{\kappa}}\leq x/v^{i\tilde{\kappa}}}
				cv^{i}m^{\tilde{\kappa}}P(\eta=m)\\
				=&cv^{i}E(\eta^{\tilde{\kappa}};\eta^{\tilde{\kappa}}\leq x/v^{i\tilde{\kappa}}).
			\end{aligned}
		\end{equation}
		
		Combining \eqref{eq3A4} and \eqref{eq3A5}, we have
		\begin{equation}\label{eq3A6}
			\begin{aligned}
				\left[\frac{P\{(\sum_{i=n}^{\infty}C_{i}1_{\{\eta_{i}^{\tilde{\kappa}}<x/v^{i\tilde{\kappa}}\}})^{\tilde{\kappa}}>x\}}{P(\eta^{\tilde{\kappa}}>x)}\right]^{1/\tilde{\kappa}}
				\leq c\sum_{i=n}^{\infty}\left[v^{i}\frac{E(\eta^{\tilde{\kappa}};\eta^{\tilde{\kappa}}\leq x/v^{i\tilde{\kappa}})}{xP(\eta^{\tilde{\kappa}}>x)}\right]^{1/\tilde{\kappa}}.
			\end{aligned}
		\end{equation}
		
		By Potter's Bound same as before, there exists $X_{1}$ such that for all $x\geq X_{1}, x/v^{i\tilde{\kappa}}\geq X_{1}$,
		\begin{align*}
			\frac{P(\eta^{\tilde{\kappa}}>x/v^{i\tilde{\kappa}})}{P(\eta^{\tilde{\kappa}}>x)}\leq A\max\{(\frac{1}{v^{i\tilde{\kappa}}})^{-\frac{\kappa}{\tilde{\kappa}}+\frac{1+\kappa-\tilde{\kappa}}{2\tilde{\kappa}}},(\frac{1}{v^{i\tilde{\kappa}}})^{-\frac{\kappa}{\tilde{\kappa}}-\frac{1+\kappa-\tilde{\kappa}}{2\tilde{\kappa}}}\}=Av^{i(\frac{\tilde{\kappa}+\kappa-1}{2})}.
		\end{align*}
		
		Using Karamata's Theorem for truncated moments (Lemma \ref{Karamata}), we have
		\begin{align*}
			\lim_{x\rightarrow\infty}\frac{E(\eta^{\tilde{\kappa}};\eta^{\tilde{\kappa}}\leq x/v^{i\tilde{\kappa}})}{xP(\eta^{\tilde{\kappa}}>x)}
			=\lim_{x\rightarrow\infty}\frac{\kappa}{\tilde{\kappa}-\kappa}\cdot\frac{P(\eta^{\tilde{\kappa}}>x/v^{i\tilde{\kappa}})}{v^{i\tilde{\kappa}}P(\eta^{\tilde{\kappa}}>x)},
		\end{align*}
		so there exists $X_{2}>X_{1}$, for all $x\geq X_{2}$, each summand on the right of (\ref{eq3A6}) can be bound by
		$$v^{i}\frac{E(\eta^{\tilde{\kappa}};\eta^{\tilde{\kappa}}\leq x/v^{i\tilde{\kappa}})}{xP(\eta^{\tilde{\kappa}}>x)}
		\leq cv^{i}\frac{P(\eta^{\tilde{\kappa}}>x/v^{i\tilde{\kappa}})}{v^{i\tilde{\kappa}}P(\eta^{\tilde{\kappa}}>x)}
		\leq
		cv^{i(\frac{\kappa+1-\tilde{\kappa}}{2})},$$
		and by first letting $x\rightarrow\infty$ then $n\rightarrow\infty$ in (\ref{eq3A6}), we get, the second term in (\ref{eq3A2}) tends to zero also.

	\end{proof}
	
	\begin{lemma}[Basrak et.al\cite{Basrak}]\label{lemmaB}
		Assume (\ref{conditionB1})-(\ref{conditionB4}) are satisfied, then as $x\rightarrow\infty$,
		$$P(X>x)\sim\frac{1}{1-\alpha^{\kappa}}\left(\frac{\beta}{1-\alpha}+p\right)P(\xi>x),$$
		or equivalently, when $p>0$,
		$$P(X>x)\sim\frac{1}{1-\alpha^{\kappa}}\left(\frac{p^{-1}\beta}{1-\alpha}+1\right)P(\eta>x).$$
	\end{lemma} 

	\section{Large deviation of partial sum: proof of Theorem \ref{theoremA} and \ref{theoremB}}
	
	\subsection{Decomposition of $S_{n}$}
	For convenience, we use the notation
	\begin{equation*}
		\Pi_{i,j}=\left\{
		\begin{aligned}
			&\theta_{j}\circ\theta_{j-1}\circ\cdots\circ\theta_{i}, \quad i\leq j,\\
			&1, \quad\quad\quad\quad\quad\quad\quad\quad i>j.
		\end{aligned}
		\right.
	\end{equation*}
	
	Then
	$$X_{n}=\sum_{i=1}^{n}\Pi_{i+1,n}\circ\eta_{i},$$
	and
	\begin{equation}\label{decomSn1}
		\begin{aligned}
			S_{n}&=X_{1}+X_{2}+\cdots+X_{n}\\
			&=\sum_{m=1}^{n}\sum_{i=1}^{m}\Pi_{i+1,m}\circ\eta_{i}\\
			&=\sum_{i=1}^{n}\sum_{m=i}^{n}\Pi_{i+1,m}\circ\eta_{i}\\
			&:=Y_{1}+\cdots+Y_{n-1}+Y_{n},
		\end{aligned}
	\end{equation}
	where  $Y_{i}=\sum_{m=i}^{n}\Pi_{i+1,m}\circ\eta_{i} \,(1\leq i\leq n)$.
	
	Then $Y_{1},\cdots,Y_{n}$ are independent and for each $1\leq i \leq n$, $Y_{i}$ has the same distribution as the total population up to $(n-i)$th generation of the underlying branching process $\{Z_{n}\}$, with $Z_{0}\overset{\text{d}}{=}\eta$.
	Precisely, using the branching and stationary property, we have, for $1\leq i\leq n$,
	\begin{equation}\label{decom1i}
		Y_{i}\overset{\text{d}}{=}T_{n-i}^{(1)}+T_{n-i}^{(2)}+\cdots+T_{n-i}^{(\eta_{i})},
	\end{equation}
	where $\{T_{n}^{(m)}\}_{m}$ are independent and have the same distribution as $T_{n}$ defined in (\ref{defTn}).\\
	
	Also we can write
	\begin{equation}\label{decomSn2}
		\begin{aligned}
			S_{n}&=\sum_{i=1}^{n}\sum_{m=i}^{n}\Pi_{i+1,m}\circ\eta_{i}\\
			&=\sum_{i=1}^{n}\sum_{m=i}^{\infty}\Pi_{i+1,m}\circ\eta_{i}
			-\sum_{i=1}^{n}\sum_{m=n+1}^{\infty}\Pi_{i+1,m}\circ\eta_{i}\\			&:=S_{n,1}-S_{n,2}.
		\end{aligned}
	\end{equation}
	
	The first term of the right hand in (\ref{decomSn2}) is
	\begin{align*}
		S_{n,1}
		&=\sum_{m=1}^{\infty}\Pi_{2,m}\circ\eta_{1}+\cdots+\sum_{m=n}^{\infty}\Pi_{n+1,m}\circ\eta_{n}\\
		&:=Y_{1}^{(\infty)}+\cdots+Y_{n}^{(\infty)},
	\end{align*}
	where $\{Y_{i}^{(\infty)}\}_{i=1}^{n}$ are independent and have the same distribution as 
	\begin{equation}\label{decom1}
		Y^{(\infty)}\overset{\text{d}}{=}T^{(1)}+T^{(2)}+\cdots+T^{(\eta)},
	\end{equation}
	with $\{T^{(m)}\}_{m}$ being independent and having the same distribution as $T$ defined in (\ref{defT}).

	The second term  of the right hand in (\ref{decomSn2}) is	
	\begin{align*}
		S_{n,2}
		:&=\sum_{i=1}^{n}(\sum_{m=n+1}^{\infty}\Pi_{n+1,m}\circ\Pi_{i+1,n}\circ\eta_{i})\\
		&=\sum_{m=n+1}^{\infty}\Pi_{n+1,m}\circ(\sum_{i=1}^{n}\Pi_{i+1,n}\circ\eta_{i})\\
		&=\sum_{m=n+1}^{\infty}\Pi_{n+2,m}\circ\theta_{n+1}\circ(\sum_{i=1}^{n}\Pi_{i+1,n}\circ\eta_{i}).
	\end{align*}
	
	So 
	\begin{equation}\label{decom2}
		S_{n,2}
		\overset{\text{d}}{=}T^{(1)}+T^{(2)}+\cdots+T^{(\theta\circ X_{n})},
	\end{equation}
	with $\{T^{(m)}\}_{m}$ being independent and having the same distribution as $T$ defined in (\ref{defT}).

	\subsection{Regular variation of partial sum }
	
	By \eqref{decomSn1}, for fixed $n\in\mathbb{N_{+}}$, $S_{n}$ is consist of $n$ independent regularly varying random variables, so we can identify the regular variation of $S_{n}$ by that of $Y_{i}(1\leq i \leq n)$.

	\begin{proposition}\label{theoremA0}
		Assume (\ref{conditionA1})-(\ref{conditionA3}) are satisfied, then $\forall n\in\mathbb{N_{+}}$,
		\begin{equation*}
			\lim_{x\rightarrow\infty}\frac{P(S_{n}>x)}{P(\eta>x)}
			=\sum_{i=1}^{n}[(\sum_{m=0}^{i-1}\alpha^{m})^{\kappa}].
		\end{equation*}
	\end{proposition}
	
	\begin{proof} 
		When $\kappa\in(0,1)$, $ET_{i-1}<\infty$. When $\kappa\geq1$, it is shown in Lemma \ref{lemmaGW1} that $ET_{i-1}^{\kappa+\delta}<\infty$, so by \eqref{decom1i}, Lemma \ref{lemmaA} and Lemma \ref{refA2}, for $1\leq i\leq n$,
		\begin{equation*}
			P(Y_{n-i+1}>x)\sim (ET_{i-1})^{\kappa}P(\eta>x)=(\sum_{m=0}^{i-1}\alpha^{m})^{\kappa}P(\eta>x).
		\end{equation*}
		
		Then the tail distribution of $S_{n}$ followed by Lemma \ref{refA1} and \eqref{decomSn1}.
	\end{proof}
	
	\begin{proposition}\label{theoremB0}
		Assume (\ref{conditionB1})-(\ref{conditionB4}) are satisfied, then $\forall n\in\mathbb{N_{+}}$,
		\begin{equation*}
			\lim_{x\rightarrow\infty}\frac{P(S_{n}>x)}{P(\xi>x)}	=\beta\sum_{i=1}^{n}\sum_{m=0}^{i-2}\alpha^{m}(\frac{1-\alpha^{i-1-m}}{1-\alpha})^{\kappa}+p\sum_{i=1}^{n}(\sum_{m=0}^{i-1}\alpha^{m})^{\kappa}.
		\end{equation*}
	\end{proposition}
	
	\begin{proof}
		For $1\leq i\leq n$, similarly,
		\begin{align*}
			P(Y_{n-i+1}>x)
			&\sim E\eta P(T_{i-1}>x)+ P(\eta>\frac{x}{ET_{i-1}})\\
			&\sim\beta\sum_{m=0}^{i-2}\alpha^{m}(\frac{1-\alpha^{i-1-m}}{1-\alpha})^{\kappa}P(\xi>x)
			+p(\sum_{m=0}^{i-1}\alpha^{m})^{\kappa}P(\xi>x).
		\end{align*}
		
		Then the tail distribution of $S_{n}$ followed by Lemma \ref{refA1} and \eqref{decomSn1}.
	\end{proof}
	
	\subsection{Proof of Theorem \ref{theoremA} and Theorem \ref{theoremB}}	
	
	We start with the proof of \eqref{eqtheoremA1} and \eqref{eqtheoremB1}.
	
	Recall the decomposition of $S_{n}$ in (\ref{decomSn2}) and observe that, for $\forall$ small $\varepsilon>0$,
	\begin{align*}
		&P(S_{n,1}-d_{n,1}>(1+\varepsilon)x)-P(S_{n,2}-d_{n,2}>\varepsilon x)\\
		\leq &P\{(S_{n,1}-d_{n,1})-(S_{n,2}-d_{n,2})>x\}\\
		\leq &P(S_{n,1}-d_{n,1}>(1-\varepsilon)x)+P(-S_{n,2}+d_{n,2}>\varepsilon x),
	\end{align*}
	where for $1\leq i\leq2$,
	\begin{equation*}
		d_{n,i}=\left\{
		\begin{aligned}
			\begin{array}{cl}
				0, &\kappa\in(0,1],\\
				ES_{n,i}, &\kappa\in(1,\infty).
			\end{array}
		\end{aligned}
		\right.
	\end{equation*}
	
	Define
	\begin{equation*}
		\left\{
		\begin{aligned}
			&I_{1}(x):=P(S_{n,1}-d_{n,1}>(1+\varepsilon)x)\\
			&I_{2}(x):=P(S_{n,1}-d_{n,1}>(1-\varepsilon)x)\\
			&I_{3}(x):=P(S_{n,2}-d_{n,2}>\varepsilon x)\\
			&I_{4}(x):=P(-S_{n,2}+d_{n,2}>\varepsilon x).
		\end{aligned}
		\right.
	\end{equation*}

	Then for $\forall$ small $\varepsilon>0$,
	\begin{align*}
		I_{1}(x)-I_{3}(x)\leq P(S_{n}-d_{n}>x) \leq I_{2}(x)+I_{4}(x).
	\end{align*}
	
	We will prove the following two propositions in what follows,
	\begin{proposition}[Estimation of $S_{n,1}$]\label{prop1}
		For $i=1,2$,
		
		(i) If (\ref{conditionA1})-(\ref{conditionA3}) are satisfied, then
		\begin{equation*}
			\lim_{\varepsilon\rightarrow0}\lim_{n\rightarrow\infty}\sup_{x\geq x_{n}}\left|\frac{I_{i}(x)}{nP(\eta>x)}-\frac{1}{(1-\alpha)^{\kappa}}\right|=0.
		\end{equation*}
		
		(ii) If (\ref{conditionB1})-(\ref{conditionB4}) are satisfied, then
		\begin{equation*}
			\lim_{\varepsilon\rightarrow0}\lim_{n\rightarrow\infty}\sup_{x\geq x_{n}}\left|\frac{I_{i}(x)}{nP(\xi>x)}-\frac{\beta+p(1-\alpha)}{(1-\alpha)^{\kappa+1}}\right|=0.
		\end{equation*}
		
	\end{proposition}
	
	\begin{proposition}[Estimation of $S_{n,2}$]\label{prop2}
		For $i=3,4,$ and $\forall\varepsilon$,
		
		(i) If (\ref{conditionA1})-(\ref{conditionA3}) are satisfied, then
		\begin{equation*}
			\lim_{n\rightarrow\infty}\sup_{x\geq x_{n}}\frac{I_{i}(x)}{nP(\eta>x)}=0.
		\end{equation*}
		
		(ii) If (\ref{conditionB1})-(\ref{conditionB4}) are satisfied, then
		\begin{equation*}
			\lim_{n\rightarrow\infty}\sup_{x\geq x_{n}}\frac{I_{i}(x)}{nP(\xi>x)}=0.
		\end{equation*}
	\end{proposition}

	Then we can conclude that if (\ref{conditionA1})-(\ref{conditionA3}) are satisfied, then
	\begin{align*}
		0
		\leq&
		\varliminf_{n\rightarrow\infty}\sup_{x\geq x_{n}}\left|
		\frac{P(S_{n}-d_{n}>x)}{nP(\eta>x)}-\frac{1}{(1-\alpha)^{\kappa}}\right|\\
		\leq&
		\varlimsup_{n\rightarrow\infty}\sup_{x\geq x_{n}}\left|
		\frac{P(S_{n}-d_{n}>x)}{nP(\eta>x)}-\frac{1}{(1-\alpha)^{\kappa}}\right|\\
		\leq&
		\max\left\{\lim_{\varepsilon\rightarrow0}\lim_{n\rightarrow\infty}\sup_{x\geq   x_{n}}
		\left(\left|\frac{I_{1}(x)}{nP(\eta>x)}-\frac{1}{(1-\alpha)^{\kappa}}\right|
		+\left|\frac{I_{3}(x)}{nP(X>x)}\right|\right), \right.\\
		&\left.\qquad\;\;
		\lim_{\varepsilon\rightarrow0}\lim_{n\rightarrow\infty}\sup_{x\geq   x_{n}}
		\left(\left|\frac{I_{2}(x)}{nP(\eta>x)}-\frac{1}{(1-\alpha)^{\kappa}}\right|
		+\left|\frac{I_{4}(x)}{nP(\eta>x)}\right|\right)
		\right\}=0,
	\end{align*}
	which completes the proof of \eqref{eqtheoremA1}, and the proof of \eqref{eqtheoremB1} is the same.
	
	The proof of \eqref{eqtheoremA2} and \eqref{eqtheoremB2} follows by analogous arguments, just need to mention that for $\forall$ small $\varepsilon>0$,
	\begin{align*}
		&P(S_{n,1}-d_{n,1}\leq-(1+\varepsilon)x)-P(S_{n,2}-d_{n,2}\leq-\varepsilon x)\\
		\leq &P\{(S_{n,1}-d_{n,1})-(S_{n,2}-d_{n,2})\leq-x\}\\
		\leq & P(S_{n,1}-d_{n,1}\leq-(1-\varepsilon)x)+P(-S_{n,2}+d_{n,2}\leq-\varepsilon x).
	\end{align*}

	$\hfill\square$
	
	In rest part of this section, we will focus on the proof of the propositions.
	
	\subsection{Proof of proposition \ref{prop1}}

	\noindent\textbf{Proof of (i)} Since under (\ref{conditionA1}) - (\ref{conditionA3}), $E(T^{\kappa+\delta})<\infty$ for $\kappa\geq1$, $ET=(1-\alpha)^{-1}$ and $\eta$ is regularly varying with index $\kappa$, we have 
	\begin{equation*}
		P(Y^{(\infty)}>x)=P(\sum_{i=1}^{\eta}T^{(i)}>x)\sim \frac{1}{(1-\alpha)^{\kappa}}\cdot P(\eta>x).
	\end{equation*}	 
	Then for any sequence $a_{n}\rightarrow\infty$,
	\begin{equation}\label{eq4A1}
		\lim_{n\rightarrow\infty}\sup_{x\geq a_{n}}\left|\frac{P(Y^{(\infty)}>x)}{P(\eta>x)}-\frac{1}{(1-\alpha)^{\kappa}}\right|=0.
	\end{equation}		
	
	In fact, $\forall \varepsilon>0$, $\exists x(\varepsilon)>0$, 
	$$\sup_{x\geq x(\varepsilon)}\left|\frac{P(Y^{(\infty)}>x)}{P(\eta>x)}-\frac{1}{(1-\alpha)^{\kappa}}\right|<\varepsilon,$$
	and since $a_{n}\rightarrow\infty$, we can choose $N(\varepsilon)$ such that $\forall n>N$, $a_{n}>x(\varepsilon)$ and thus
	$$\sup_{x\geq a_{n}}\left|\frac{P(Y^{(\infty)}>x)}{P(\eta>x)}-\frac{1}{(1-\alpha)^{\kappa}}\right|\leq\sup_{x\geq x(\varepsilon)}\left|\frac{P(Y^{(\infty)}>x)}{P(\eta>x)}-\frac{1}{(1-\alpha)^{\kappa}}\right|<\varepsilon,$$
	which means (\ref{eq4A1}).

	Recall that $$S_{n,1}=Y_{1}^{(\infty)}+\cdots+Y_{n}^{(\infty)}$$
	is the summation of $n$ i.i.d random variables, using Theorem \ref{refB} where $p=1, q=0$, we have
	\begin{equation}\label{eq4A2}
		\lim_{n\rightarrow\infty}\sup_{x\geq x_{n}}
		\left|\frac{P(S_{n,1}-d_{n,1}>x)}{nP(Y^{(\infty)}>x)}-1\right|=0,
	\end{equation}
	and
	$$\lim_{n\rightarrow\infty}\sup_{x\geq x_{n}}
	\frac{P(S_{n,1}-d_{n,1}\leq -x)}{nP(Y^{(\infty)}>x)}=0,$$
	where for $\kappa\in(0,2]$, one can choose $x_{n}=n^{\delta+1/\kappa}$ for any $\delta>0$; for $\kappa\in(2,\infty)$, one can choose $x_{n}=\sqrt{an\log n}$ for $a>\kappa-2$.
	We have
	\begin{align*}
		0\leq &\varliminf_{n\rightarrow\infty}\sup_{x\geq x_{n}}
		\left|\frac{P(S_{n,1}-d_{n,1}>x)}{nP(\eta>x)}-\frac{1}{(1-\alpha)^{\kappa}}\right|\\
		\leq &\varlimsup_{n\rightarrow\infty}\sup_{x\geq x_{n}}
		\left|\frac{P(S_{n,1}-d_{n,1}>x)}{nP(\eta>x)}-\frac{1}{(1-\alpha)^{\kappa}}\right|\\
		=&\varlimsup_{n\rightarrow\infty}\sup_{x\geq x_{n}}
		\left|\frac{P(S_{n,1}-d_{n,1}>x)}{nP(Y^{(\infty)}>x)}\cdot\frac{P(Y^{(\infty)}>x)}{P(\eta>x)}-\frac{1}{(1-\alpha)^{\kappa}}\right|\\
		\leq&\varlimsup_{n\rightarrow\infty}\sup_{x\geq x_{n}}\left|\frac{P(S_{n,1}-d_{n,1}>x)}{nP(Y^{(\infty)}>x)}-1\right|\cdot
		\sup_{x\geq x_{n}}\frac{P(Y^{(\infty)}>x)}{P(\eta>x)}\\
		+&\varlimsup_{n\rightarrow\infty}\sup_{x\geq x_{n}}\left|\frac{P(Y^{(\infty)}>x)}{P(\eta>x)}-\frac{1}{(1-\alpha)^{\kappa}}\right|=0,
	\end{align*}
	by equation (\ref{eq4A1}) and (\ref{eq4A2}).
	So
	$$\lim_{n\rightarrow\infty}\sup_{x\geq x_{n}}
	\left|\frac{P(S_{n,1}-d_{n,1}>x)}{nP(\eta>x)}-\frac{1}{(1-\alpha)^{\kappa}}\right|=0,$$
	and similarly,
	$$\lim_{n\rightarrow\infty}\sup_{x\geq x_{n}}			\left|\frac{P(S_{n,1}-d_{n,1}\leq -x)}{nP(\eta>x)}\right|=0.$$	
	
	Then by the regular variation of $\eta$ and the same discussion as (\ref{eq4A1}), we have
	\begin{align*}
		&\varlimsup_{n\rightarrow\infty}\sup_{x\geq x_{n}}\left|\frac{I_{1}(x)}{nP(\eta>x)}-\frac{1}{(1-\alpha)^{\kappa}}\right|\\
		=&\varlimsup_{n\rightarrow\infty}\sup_{x\geq x_{n}}\left|\frac{P(S_{n,1}-d_{n,1}>(1-\varepsilon)x)}{nP(\eta>(1-\varepsilon)x)}\cdot\frac{P(\eta>(1-\varepsilon)x)}{P(\eta>x)}-\frac{1}{(1-\alpha)^{\kappa}}\right|\\
		\leq&\varlimsup_{n\rightarrow\infty}\sup_{x\geq x_{n}}\left|\frac{P(S_{n,1}-d_{n,1}>(1-\varepsilon)x)}{nP(\eta>(1-\varepsilon)x)}-\frac{1}{(1-\alpha)^{\kappa}}\right|\cdot\sup_{x\geq x_{n}}\frac{P(\eta>(1-\varepsilon)x)}{P(\eta>x)}\\
		+&\frac{1}{(1-\alpha)^{\kappa}}\varlimsup_{n\rightarrow\infty}\sup_{x\geq x_{n}}\left|\frac{P(\eta>(1-\varepsilon)x)}{P(\eta>x)}-1\right|
		\rightarrow 0
	\end{align*}
	as $\varepsilon\rightarrow0$, so Proposition \ref{prop1} holds for $i=1$, the proof for $i=2$ is  similar and we omit the details. 	\\
	
	\noindent\textbf{Proof of (ii)} Since under (\ref{conditionB1}) - (\ref{conditionB4}), $T$ is regularly varying with index $\kappa$, $ET=(1-\alpha)^{-1}$ and the tail of $\eta$ is not-heavier than $T$, we have
	\begin{align*}
		P(Y^{(\infty)}>x)
		&\sim E\eta P(T>x)+P(\eta>\frac{x}{ET})\\
		&\sim\frac{\beta}{(1-\alpha)^{\kappa+1}}P(\xi>x)+\frac{p}{(1-\alpha)^{\kappa}}P(\xi>x)\\
		&=\frac{\beta+p(1-\alpha)}{(1-\alpha)^{\kappa+1}}P(\xi>x).
	\end{align*}		
	
	Then the rest part of the proof is similar as (i).
	
	$\hfill\square$
	
	\subsection{Proof of proposition \ref{prop2}}
	\noindent\textbf{Proof of (i)} 
	Using the same notations as before, define 
	$$S^{(\infty)}:=T^{(1)}+T^{(2)}+\cdots+T^{(\theta\circ X)}.$$
	
	By Lemma \ref{lemmaA} and  Lemma \ref{refA2} we have
	$$P(\theta\circ X>x)=P(\sum_{i=1}^{X}\xi_{i}>x)\sim\alpha^{\kappa}P(X>x),$$
	then for independent $\theta\circ  X$ and $\{T^{(n)}\}_{n}$,
	$$P(S^{(\infty)}>x)\sim P(\theta\circ X>\frac{x}{ET})\sim\frac{\alpha^{\kappa}}{(1-\alpha)^{\kappa}(1-\alpha^{\kappa})}P(\eta>x).$$
	
	Thus for any sequence $a_{n}\rightarrow\infty$,
	\begin{equation}\label{eq4A4}
		\lim_{n\rightarrow\infty}\sup_{x\geq a_{n}}\frac{P(S^{(\infty)}>x)}{P(\eta>x)}=\frac{\alpha^{\kappa}}{(1-\alpha)^{\kappa}(1-\alpha^{\kappa})}.
	\end{equation}
	
	Since as $n\rightarrow\infty$, recall the decomposition \eqref{decom2},
	\begin{align*}
		S_{n,2}
		&\overset{\text{d}}{=}T^{(1)}+T^{(2)}+\cdots+T^{(\theta\circ X_{n})}
		\overset{\text{d}}{\rightarrow}S^{(\infty)},
	\end{align*}
	we have, for $\forall x$,
	$$\lim_{n\rightarrow\infty}P(S_{n,2}>x)=P(S^{(\infty)}>x).$$
	
	Thus $\exists N\in\mathbb{N_{+}}, \forall n>N,$
	$$P(S_{n,2}>x)\leq 2P(S^{(\infty)}>x),$$
	and
	\begin{align*}
		\sup_{x\geq x_{n}}\frac{P(S_{n,2}>\varepsilon x)}{nP(\eta>x)}
		&\leq\sup_{x\geq x_{n}}\frac{2P(S^{(\infty)}>\varepsilon x)}{nP(\eta>x)}\\
		&\leq\frac{2}{n}\cdot\sup_{x\geq x_{n}}\frac{P(S^{(\infty)}>\varepsilon x)}{P(S^{(\infty)}>x)}\cdot\frac{P(S^{(\infty)}>x)}{P(\eta>x)},
	\end{align*}
	which means
	\begin{equation*}
		\lim_{n\rightarrow\infty}\sup_{x\geq x_{n}}\frac{P(S_{n,2}>\varepsilon x)}{nP(\eta>x)}=0
	\end{equation*}
	for $\kappa\in(0,1]$.
	
	For $\kappa>1$, define
	$$ES_{n,2}=E(\theta\circ X_{n})\cdot ET=\beta\alpha\cdot\frac{1-\alpha^{n}}{1-\alpha}\cdot\frac{1}{1-\alpha}\uparrow \frac{\beta\alpha}{(1-\alpha)^{2}}:=m,$$
	then similarly, using equation (\ref{eq4A4}), for $\forall n> N$,
	\begin{align*}
		&\sup_{x\geq x_{n}}\frac{\max\{I_{3}(x),I_{4}(x)\}}{nP(\eta>x)}\\
		\leq&\sup_{x\geq x_{n}}\frac{P(S_{n,2}+ES_{n,2}>\varepsilon x)}{nP(\eta>x)}\\
		\leq&\sup_{x\geq x_{n}}\frac{2P(S^{(\infty)}>\varepsilon x-m)}{nP(\eta>x)}\\
		\leq&\frac{2}{n}\cdot\sup_{x\geq x_{n}}\frac{P(S^{(\infty)}>\varepsilon x-m)}{P(\eta>\varepsilon x-m)}\cdot \sup_{x\geq x_{n}}\frac{P(\eta>\varepsilon x-m)}{P(\eta>x)}\rightarrow0,
	\end{align*}
	as $n\rightarrow\infty$.\\
	
	\noindent\textbf{Proof of (ii)}	
	By Lemma \ref{lemmaB} and  Lemma \ref{refA4} we have
	$$P(\theta\circ X>x)\sim\alpha^{\kappa}P(X>x)+(EX)P(\xi>x)
	\sim\frac{1}{1-\alpha^{\kappa}}\left(\frac{\beta}{1-\alpha}+p\alpha^{\kappa}\right)P(\xi>x),$$
	then for independent $X$ and $\{T^{(n)}\}_{n}$,	
	$$P(S^{(\infty)}>x)\sim E(\theta\circ X)\cdot P(T>x)+P\left(\theta\circ X>\frac{x}{ET}\right),$$
	thus for any sequence $a_{n}\rightarrow\infty$,
	\begin{equation}
		\lim_{n\rightarrow\infty}\sup_{x\geq a_{n}}\frac{P(S^{(\infty)}>x)}{P(\xi>x)}=\frac{\beta\alpha}{(1-\alpha)^{\kappa+2}}+\frac{1}{(1-\alpha)^{\kappa}(1-\alpha^{\kappa})}\left(\frac{\beta}{1-\alpha}+p\alpha^{\kappa}\right).
	\end{equation}
	
	Then the rest part of the proof is same as (i).
	$\hfill\square$

	\appendix
	\begin{appendices}
		\section{Properties of regularly varying distribution}

		\begin{lemma}[Potter's Bound \cite{Bingham}]\label{Potter}
			If $\xi$ is regularly varying with index $\kappa\geq0$, then for any chosen $A>1$, $\delta>0$, there exists $X=X(A,\delta)$ such that for all $x\geq X, y\geq X$,
			$$\frac{P(\xi>y)}{P(\xi>x)}\leq A\max\{(\frac{y}{x})^{-\kappa+\delta},(\frac{y}{x})^{-\kappa-\delta}\}.$$
		\end{lemma}
		
		\begin{lemma}[Karamata's Theorem For Truncated Moments \cite{Bingham}]\label{Karamata}
			If $\xi$ is regularly varying with index $\kappa>0$, then for $\forall\tilde{\kappa}>\kappa$,
			$$\lim_{x\rightarrow\infty}\frac{E(\xi^{\tilde{\kappa}};\xi\leq x)}{x^{\tilde{\kappa}}P(\xi>x)}=\frac{\kappa}{\tilde{\kappa}-\kappa},$$
		\end{lemma}
		
		\begin{lemma}[Davis and Resnick \cite{Davis}]\label{refA1}
			Suppose $Y_{1},\cdots,Y_{n}$ are nonnegative random variables(but not necessarily independent or identically distributed).
			If $Y_{1}$ is regularly varying with index $\kappa>0$ and
			\begin{equation*}
				\lim_{x\rightarrow\infty}\frac{P(Y_{i}>x)}{P(Y_{1}>x)}= c_{i}, \quad i=1,2,...,n,
			\end{equation*}
			\begin{equation*}
				\lim_{x\rightarrow\infty}\frac{P(Y_{i}>x, Y_{j}>x)}{P(Y_{1}>x)}= 0, \quad i\neq j,
			\end{equation*}
			then
			\begin{equation*}
				\lim_{x\rightarrow\infty}\frac{P(\sum_{i=1}^{n}Y_{i}>x)}{P(Y_{1}>x)}=c_{1}+c_{2}+\cdots+c_{n}.
			\end{equation*}
		\end{lemma}
		
		\begin{lemma}[Fay et.al \cite{Fay}, Robert et.al\cite{Robert}, Barczy et.al\cite{Barczy}]\label{refA2}
			Assume that $\eta$ is a non-negative integer-valued random variable, $\{\xi_{i}\}$ is an i.i.d non-negative sequence and independent of $\eta$, such that 
			
			(i) $\eta$ is regularly varying with index $\kappa>0$;
			
			(ii) $0<E\xi<\infty$;
			
			(iii) if $\kappa\geq1$ assume additionally that there exits $\delta>0$ with $E\xi^{\kappa+\delta}<\infty$. 
			
			Then as $x\rightarrow\infty$,
			$$P(\sum_{i=1}^{\eta}\xi_{i}>x)\sim P(\eta>\frac{x}{E\xi})\sim (E\xi)^{\kappa}P(\eta>x).$$
		\end{lemma}
		
		\begin{lemma}[Fay et.al \cite{Fay}]\label{refA3}
			Assume that $\eta$ is a non-negative integer-valued random variable, $\{\xi_{i}\}$ is an i.i.d non-negative sequence and independent of $\eta$, such that 
			
			(i) $\xi$ is regularly varying with index $\kappa>0$;
			
			(ii) $0<E\eta<\infty$;
			
			(iii) if $\kappa\geq1$ assume additionally that there exits $\delta>0$ with  $E\eta^{\kappa+\delta}<\infty$. 
			
			Then as $x\rightarrow\infty$,
			$$P(\sum_{i=1}^{\eta}\xi_{i}>x)\sim E\eta P(\xi>x).$$
		\end{lemma}
		
		\begin{lemma}[Denisov et.al\cite{Denisov}, Fay et.al \cite{Fay}]\label{refA4}
			Assume that $\eta$ is a non-negative integer-valued random variable, $\{\xi_{i}\}$ is an i.i.d non-negative sequence and independent of $\eta$, such that 
			
			(i) $\xi$ is regularly varying with index $\kappa\geq1$;
			
			(ii) $0<E\xi<\infty$, $0<E\eta<\infty$;
			
			(iii) there exists positive constant $c$ that $P(\eta>x)\sim cP(\xi>x)$.
			
			Then as $x\rightarrow\infty$,
			$$P(\sum_{i=1}^{\eta}\xi_{i}>x)\sim E\eta P(\xi>x)+ P(\eta>\frac{x}{E\xi})\sim(E\eta+c(E\xi)^{\kappa})P(\xi>x).$$
		\end{lemma}
		
		\begin{lemma}[Asmussen and Foss \cite{Asmussen}]\label{refA5}
			Consider equation
			$$T\overset{\text{d}}{=}Q+\sum_{i=1}^{N}T_{m},$$
			where $Q$ and $N$ are (possibly dependent) nonnegative, nondegenerate r.v.'s, $N$ is integer valued, $\{T_{n}\}_{n}$ is i.i.d and distributed as $T$.
			
			If 
			$$EN<1, \quad EQ<\infty,$$
			then there is only one nonnegative solution $T$ with finite mean, and $ET=\frac{EQ}{1-EN}$.
			
			If further, for $\forall \varepsilon$ small, c $\in (ET-\varepsilon, ET+\varepsilon)$, the distribution of $Q+cN$ is intermediate regularly varying, then as $x\rightarrow\infty$,
			$$P(T>x)\sim\frac{1}{1-EN}P(Q+ET\cdot N>x).$$   
			
		\end{lemma}
		
		\section{Large deviations for sums of i.i.d. sequences}
		\begin{theorem}[A.V. Nagaev\cite{AV}, S.V. Nagaev\cite{SV}, Cline and Hsing \cite{Cline}]\label{refB}
			Suppose ${X_{n}}$ is an i.i.d sequence which is regularly varying with index $\kappa>0$ and $p+q=1$ such that
			$$P(X>x)\sim p\frac{L(x)}{x^{\kappa}}, \quad P(X\leq -x)\sim q\frac{L(x)}{x^{\kappa}}, \qquad as \quad x\rightarrow\infty.$$
			
			Then the following relations hold for suitable sequences $\{x_{n}\}\uparrow\infty$ :
			$$\lim_{n\rightarrow\infty}\sup_{x\geq x_{n}}\left|\frac{P(\sum_{i=1}^{n}X_{i}-d_{n}>x)}{nP(|X|>x)}-p\right|=0,$$
			$$\lim_{n\rightarrow\infty}\sup_{x\geq x_{n}}\left|\frac{P(\sum_{i=1}^{n}X_{i}-d_{n}\leq-x)}{nP(|X|>x)}-q\right|=0,$$
			where
			\begin{equation*}
				d_{n}=\left\{
				\begin{aligned}
					&0, \qquad\qquad\quad \kappa\in(0,1]\\
					&E(\sum_{i=1}^{n}X_{i}), \quad \kappa\in(1,\infty)
				\end{aligned}
				\right.
			\end{equation*}
			and if $\kappa\in(0,2]$, one can choose $x_{n}=n^{\delta+1/\kappa}$ for any $\delta>0$; if $\kappa\in(2,\infty)$, one can choose $x_{n}=\sqrt{an\log n}$ for $a>\kappa-2$.
		\end{theorem}
	\end{appendices}

\end{document}